\documentclass[12pt,reqno]{amsart}
\usepackage[colorlinks,citecolor=red,pagebackref,hypertexnames=false]{hyperref}

\makeatletter
\def\@tocline#1#2#3#4#5#6#7{\relax
  \ifnum #1>\c@tocdepth % then omit
  \else
    \par \addpenalty\@secpenalty\addvspace{#2}%
    \begingroup \hyphenpenalty\@M
    \@ifempty{#4}{%
      \@tempdima\csname r@tocindent\number#1\endcsname\relax
    }{%
      \@tempdima#4\relax
    }%
    \parindent\z@ \leftskip#3\relax \advance\leftskip\@tempdima\relax
    \rightskip\@pnumwidth plus4em \parfillskip-\@pnumwidth
    #5\leavevmode\hskip-\@tempdima
      \ifcase #1
       \or\or \hskip 1em \or \hskip 2em \else \hskip 3em \fi%
      #6\nobreak\relax
    \dotfill\hbox to\@pnumwidth{\@tocpagenum{#7}}\par
    \nobreak
    \endgroup
  \fi}
\makeatother

\usepackage{tikz,amsthm,amsmath,amstext,amssymb,amscd,epsfig,euscript, mathrsfs, dsfont,pspicture,multicol,graphpap,graphics,graphicx,times,enumerate,subfig,sidecap,wrapfig,color}
\usepackage{ hyperref}
\usepackage{enumerate}

 \numberwithin{equation}{section}

\def\bR{{\mathbb{R}}}

\def\bZ{{\mathbb{Z}}}

\def\bN{{\mathbb{N}}}

\def\cD{{\mathscr{D}}}

\def\cH{{\mathscr{H}}}

\def\cM{{\mathscr{M}}}

\def\cP{{\mathscr{P}}}

\def\ve{\varepsilon}

\DeclareMathOperator{\diam}{diam}
\def\dist{\mathop\mathrm{dist}} 						%distance
\def\supp{\mathop\mathrm{supp}}					%support
\newcommand{\ps}[1]{\left( #1 \right)}

\newcommand{\ck}[1]{\left\{#1 \right\}}

\newcommand{\cnj}[1]{\overline{#1}}

\def\XXint#1#2#3{{\setbox0=\hbox{$#1{#2#3}{\int}$ }
\vcenter{\hbox{$#2#3$ }}\kern-.58\wd0}}

\theoremstyle{plain}
\newtheorem{theorem}{Theorem}

\newtheorem{corollary}[theorem]{Corollary}

\newtheorem{lemma}[theorem]{Lemma}

\newtheorem{proposition}[theorem]{Proposition}

\theoremstyle{definition}

\newtheorem{definition}[theorem]{Definition}
\newtheorem{remark}[theorem]{Remark}

\numberwithin{equation}{section}
\numberwithin{theorem}{section}

\newcommand\eqn[1]{\eqref{e:#1}}

\newcommand\Theorem[1]{Theorem \ref{t:#1}}
\newcommand\Lemma[1]{Lemma \ref{l:#1}}

\newcommand\Remark[1]{Remark \ref{r:#1}}

  \DeclareFontFamily{U}{mathb}{\hyphenchar\font45} 
\DeclareFontShape{U}{mathb}{m}{n}{
      <5> <6> <7> <8> <9> <10> gen * mathb
      <10.95> mathb10 <12> <14.4> <17.28> <20.74> <24.88> mathb12
      }{}
\DeclareSymbolFont{mathb}{U}{mathb}{m}{n}

\DeclareMathSymbol{\toitself}      {3}{mathb}{"FD}  %*

\begin{document}

\title[A characterization of $1$-rectifiable doubling measures]{A characterization of $1$-rectifiable doubling measures with connected supports}

\author{Jonas Azzam}
\address{Departament de Matem\`atiques\\ Universitat Aut\`onoma de Barcelona \\ Edifici C Facultat de Ci\`encies\\
08193 Bellaterra (Barcelona) }
\email{jazzam "at" mat.uab.cat}
\author{Mihalis Mourgoglou}
\address{Departament de Matem\`atiques, Universitat Aut\`onoma de Barcelona and Centre de Reserca Matem\` atica, Edifici C Facultat de Ci\`encies \\ 08193 Bellaterra (Barcelona)}
\email{mmourgoglou@crm.cat}

\thanks{The authors were supported by the ERC grant 320501 of the European Research Council (FP7/2007-2013).}

\maketitle

\begin{abstract}
Garnett, Killip, and Schul have exhibited a doubling measure $\mu$ with support equal to $\mathbb{R}^{d}$ which is {\it $1$-rectifiable}, meaning there are countably many curves $\Gamma_{i}$ of finite length for which $\mu(\mathbb{R}^{d}\backslash \bigcup \Gamma_{i})=0$. In this note, we characterize when a doubling measure $\mu$ with support equal to a connected metric space $X$ has a $1$-rectifiable subset of positive measure and show this set coincides up to a set of $\mu$-measure zero with the set of $x\in X$ for which $\liminf_{r\rightarrow 0}  \mu(B_{X}(x,r))/r>0$.
\end{abstract}
\tableofcontents

\maketitle

\tableofcontents

\section{Introduction}

Recall that a Borel measure $\mu$ on a metric space $X$ is {\it doubling} if there is $C_{\mu}>0$ so that 
\begin{equation}\label{e:doubling}
\mu(B_{X}(x,2r))\leq C_{\mu}\mu(B_{X}(x,r))\mbox{ for all }x\in X, \;\; r>0.
\end{equation} 
In \cite{GKS}, Garnett, Killip, and Schul exhibit a doubling measure $\mu$ with support equal to $\bR^{n}$, $n>1$, that is $1$-rectifiable in the sense that there are countably many curves $\Gamma_{i}$ of finite length such that $\mu(\bR^{n}\backslash \bigcup \Gamma_{i})=0$. This is surprising given that such measures give zero measure to smooth or bi-Lipschitz curves in $\bR^{d}$. To see this, note that for such a curve $\Gamma$ and for each $x\in \Gamma$, there is $r_{x},\delta_{x}>0$ so that for all $r\in (0,r_{x})$ there is $B_{\bR^{d}}(y_{x,r},\delta_{x} r)\subseteq B_{\bR^{n}}(x,r_{x})\backslash \Gamma$, so by the Lebesgue differentiation theorem, $\mu(\Gamma)=0$. If $\Gamma$ is just Lipschitz and not bi-Lipschitz, however, we only know this property holds for every point in $\Gamma$ outside a set of zero length. The aforementioned result shows that Lipschitz curves of finite length can in some sense be coiled up tightly enough so that this zero length set accumulates on a set of positive doubling measure. 

The notion of rectifiability of a measure that we are using is not universal. In \cite{ADT}, a measure $\mu$ in Euclidean space being $d$-rectifiable means $\mu\ll \cH^{d}$ and $\supp \mu$ is $d$-rectifiable. In our setting, however, we don't require absolute continuity of our measures. To avoid ambiguity, we fix our definition below, which is the convention used by Federer \cite[Section 3.2.14]{Federer}. 

\begin{definition}
If $\mu$ is a Borel measure on a  metric space $X$, $d$ is an integer, and $E\subseteq X$ a Borel set, we say $E$ is {\it $(\mu,d)$-rectifiable} if $\mu(E\backslash \bigcup_{i=1}^{\infty} \Gamma_{i})=0$ where $\Gamma_{i}=f_{i}(E_{i})$, $E_{i}\subseteq \bR^{d}$, and $f_{i}:E_{i}\rightarrow X$ is Lipschitz. We say $\mu$ is {\it $d$-rectifiable} if $\supp \mu$ is $(\mu,d)$-rectifiable. 
\end{definition}

A set $E\subseteq \bR^{n}$ of positive and finite $\cH^{d}$-measure is {\it $d$-rectifiable} if if is $(\cH^{d},d)$-rectifiable (see in \cite[Definition 15.3]{Mattila} and the few paragraphs preceding it). This is also equivalent to being covered up to set of $\cH^{d}$-measure zero by Lipschitz graphs \cite[Lemma 15.4]{Mattila}. The example from \cite{GKS}, however, shows that being almost covered by Lipschitz graphs versus Lipschitz images are not equivalent definitions for rectifiability of a measure.

Since \cite{GKS}, it has been an open question to classify which doubling measures on $\mathbb{R}^{d}$ are rectifiable. Very recently, Badger and Schul have given a complete description. First, for a general Radon measure in $\bR^{d}$ and $A$ compact with $\mu(A)>0$, define
\[\beta_{2}^{(1)}(\mu,A)^{2}=\inf_{L}\int_{A}\ps{\frac{\dist(x,L)}{\diam A}}^{2}\frac{d\mu(x)}{\mu(A)}\]
where the infimum is taken over all lines $L\subseteq \bR^{d}$. 

\begin{theorem} (\cite[Corollary 1.12]{BS2}) If $\mu$ is a Radon measure on $\bR^{d}$ such that $\liminf_{r\rightarrow 0} \beta_{2}^{(1)}(\mu,B_{\bR^{d}}(x,r))>0$ for $\mu$ almost every $x\in \bR^{d}$, then $\mu$ is $1$-rectifiable if and only if 
\begin{equation}
\sum_{x\in Q\atop \ell(Q)\leq 1}\frac{\diam Q}{\mu(Q)}<\infty \;\;\mu\mbox{ a.e.}
\label{e:BS}
\end{equation}
where the sum is over half-open dyadic cubes $Q$.
\label{t:BS}
\end{theorem}

It is not hard to show that if $\mu$ is a doubling measure with $\supp \mu=\bR^{d}$, $d\geq 2$, then there is $c>0$ depending on the doubling constant such that $\beta_{2}^{(1)}(\mu,B)\geq c>0$ for any ball $B\subseteq \bR^{d}$, so the above theorem characterizes all $1$-rectifiable doubling measures with support equal to all of $\bR^{d}$. 

In this short note, we take a different approach and provide a complete classification of $1$-rectifiable doubling measures not just with support equal to $\mathbb{R}^{d}$ but with support equal to any topologically connected metric space. It turns out that the rectifiable part of such a measure coincides up to a set of $\mu$ measure zero with the set of points where the lower $1$-density is positive, where for $s>0$, we define the {\it lower $s$-density} as
\[\underline{D}^{s}(\mu,x):=\liminf_{r\rightarrow 0} \frac{\mu(B_{X}(x,r))}{r^{s}}.\]

\begin{theorem}[Main Theorem]
Let $\mu$ be a doubling measure whose support is a topologically connected metric space $X$ and let $E\subseteq X$ be compact. Then $E$ is $(\mu,1)$-rectifiable if and only if $\underline{D}^{1}(\mu,x)>0$ for $\mu$-a.e. $x\in E$.
\end{theorem}

Note that there are no other topological or geometric restrictions on $X$: the support of $\mu$ may have topological dimension two (like $\bR^{2}$ for example), yet if $\underline{D}^{1}(\mu,x)>0$ $\mu$-a.e., then $\mu$ is supported on a countable union of Lipschitz images of $\bR$. Also observe that the condition $\underline{D}^{1}(\mu,x)>0$ is a weaker condition than \eqn{BS}. An interesting corollary of the Main Theorem and \Theorem{BS} is the following:

\begin{corollary}
If $\mu$ is a doubling measure in $\bR^{d}$ with connected support such that $\liminf_{r\rightarrow 0} \beta_{2}^{(1)}(\mu,B_{\bR^{d}}(x,r))>0$ and $\underline{D}^{1}(\mu,x)>0$ $\mu$-a.e., then \eqn{BS} holds.
\end{corollary}

%The lower density condition has come up in the context of rectifiability before. In \cite{Pa} Pajot showed that if $\mu=\cH^{d}|_{E}$ is a finite measure for some compact set $E\subseteq \bR^{d}$, $\underline{D}^{n}(\mu,x)>0$ and an additional geometric condition involving $\beta$-numbers holds, then $\mu$ is $n$-rectifiable. More recently, Badger and Schul \cite{BS2} obtained the same conclusion from the same $\beta$-number condition but from the weaker hypothesis $\mu\ll \cH^{d}$. However, when $d = 1$ and $\mu$ is a doubling measure with connected support, our main result implies rectifiablity without the assumption of absolute continuity of $\mu$ and without the $\beta$-number condition in \cite{Pa} and \cite{BS2}.\\

The authors thank Raanan Schul for his encouragement and helpful discussions which improved the result, as well as John Garnett and the anonymous referee whose advice greatly improved the readability of the paper.

\section{Proof of the Main Theorem: Sufficiency}

When dealing with any metric space $X$, we will let $B_{X}(x,r)$ denote the set of points {\it in $X$} of distance less than $r>0$ from $x$. If $B=B_{X}(x,r)$ and $M>0$, we will denote $MB=B_{X}(x,Mr)$. For a Borel set $A\subseteq X$, we define the (spherical) $1$-Hausdorff measure as
\[\cH_{\delta}^{1}(A)=\inf\ck{\sum_{i=1}^{\infty} 2r_{i}:A\subseteq \bigcup_{i=1}^{\infty}B_{X}(x_{i},r_{i}), \;\; x_{i}\in A, r_{i}\in (0,\delta)}\]
and $\cH^{1}(A)=\inf_{\delta>0}\cH^{1}_{\delta}(A)$. 

For $A,B\subseteq X$ we set
\[\dist(A,B)=\inf\{|x-y|:x\in A,y\in B\}\]
and for $x\in X$, $\dist(x,A)=\dist(\{x\},A)$. \\ 

\begin{remark}\label{r:C(X)}
By the Kuratowski embedding theorem, if $X$ is separable (which happens, for example, if $X=\supp \mu$ for a locally finite measure $\mu$), $X$ is isometrically embeddable into $C(X)$, where $C(X)$ is the Banach space of bounded continuous functions on $X$ equipped with the supremum norm $|f|=\sup_{x\in X}|f(x)|$. Thus, we can assume without loss of generality that $X$ is the subset of a complete Banach space, and we will abuse notation by calling this space $C(X)$ as well, so that $X\subseteq C(X)$.
\end{remark}

The forward direction of the Main Theorem is proven for general measures in Euclidean space in \cite[Lemma 2.7]{BS1}, where in fact they prove a higher dimensional version. Below we provide a proof that works for metric spaces in the one-dimensional case.

\begin{proposition}
Let $\mu$ be a finite measure with $X:=\supp \mu$ a metric space and suppose $\mu$ is $1$-rectifiable. Then $\underline{D}^{1}(\mu,x)>0$ for $\mu$-a.e. $x\in \supp \mu$.
\end{proposition}

\begin{proof}
Let 
\[F=\{x\in \supp\mu:\underline{D}^{1}(\mu,x)=0\}\]
and let $\ve,\delta>0$. Since $\mu$ is rectifiable, there are Lipschitz functions $f_{i}:A_{i}\rightarrow X$, where $A_{i}\subseteq [0,1]$ are compact Borel sets of positive measure and $i=1,...,N$, so that 
\[ \mu\ps{E\backslash \bigcup_{i=1}^{N}f_{i}(A_{i})}<\delta.\]
We can extend each $f_{i}$ affinely on the intervals in the complement of $A_{i}$ to a Lipschitz function $f_{i}:[0,1]\rightarrow C(X)$. Let $d=\min_{i=1,...,N}\diam f_{i}([0,1])$, so that $r\in (0,d)$ and $x\in G:=\bigcup_{i=1}^{N} f_{i}([0,1])$ implies $\cH^{1}(B_{C(X)}(x,r)\cap G)\geq r$ (simply because now the images of the $f_{i}$ are connected).

For each $x\in F\cap G$, there is $r_{x}\in (0,d/5)$ so that $\mu(B_{X}(x,5r_x))<\ve r_x$. By the Vitali Covering Theorem (see \cite[Lemma 1.2]{H}), there are countably many disjoint balls balls $B_{i}=B_{X}(x_{i},r_{i})$ with centers in $F$ so that $\bigcup 5B_{i}\supseteq F$. Thus,
\begin{align*}
\mu(F\cap G)
& \leq \sum_{i} \mu(5B_{i})
\leq \ve\sum_{i} r_{i}
\leq \ve\sum_{i} \cH^{1}(B_{C(X)}(x_{i},r_{i})\cap G)\\
& \leq \ve\cH^{1}(G).\end{align*}
Thus,
\[\mu(F)<\delta+ \ve \cH^{1}(G).\]
Keeping $\delta$ (and hence $G$) fixed and sending $\ve\rightarrow 0$, we get $\mu(F)<\delta$ for all $\delta>0$ and thus $\mu(F)=0$. 
\end{proof}

\section{Proof of the Main Theorem: Necessity}

What remains is to prove the reverse direction of the Main Theorem, which we summarise in the next lemma.

\begin{lemma}\label{l:mainlemma}
Let $\mu$ be a doubling measure with constant $C_{\mu}>0$ and support $X$, a topologically connected metric space. Then $\{x\in X: \underline{D}^{1}(\mu,x)>0\}$ is $(\mu,1)$-rectifiable.
\end{lemma}

To prove \Lemma{mainlemma}, it suffices to show the following lemma.

\begin{lemma}\label{l:reduction}
Let $\mu$ be a doubling measure and support $X$,a  topologically connected complete metric space. If $E\subseteq X$ is a compact set for which $E\subseteq B_{X}(\xi_{0},r_{0}/2)$ for some $\xi_{0}\in X$, $r_{0}>0$, and 
 \begin{equation}\label{e:mu>r}
\mu(B_{X}(x,r))\geq 2r\mbox{ for all }x\in E \mbox{ and }r\in (0,r_{0}).
\end{equation}
then $E=f(A)$ for some $A\subseteq \bR$ and Lipschitz function $f:A\rightarrow X$.
\end{lemma}

\begin{proof}[Proof of \Lemma{mainlemma} using \Lemma{reduction}]

First, note that if we define $\cnj{\mu}(A)=\mu(A\cap X)$, then $\cnj{\mu}$ is a doubling measure on $\cnj{X}$, where the closure is in $C(X)$ (recall \Remark{C(X)}). Moreover, the closure $\cnj{X}$ is still topologically connected but now is a complete metric space since $C(X)$ is complete. Thus, for proving \Lemma{mainlemma}, we can assume without loss of generality that $X$ is complete.

Let $F:=\{x\in X: \underline{D}^{1}(\mu,x)>0\}$.  For $j,k\in \bN$, let
\[F_{j,k}=\{x\in F: \mu(B_{X}(x,r))\geq r/j\mbox{ for }0<r<k^{-1}\}.\]
Then $F=\bigcup_{j,k\in \bN}F_{j,k}$. Furthermore, we can write $F_{j,k}$ as a countable union of sets $\{F_{j,k,\ell}\}_{\ell\in \bN}$ with diameters less than $\frac{1}{3k}$. It suffices then to show that each one of these sets is $1$-rectifiable. Fix $j,k,\ell\in \bN$. Then the measure $j\mu$ and the set $F_{j,k,\ell}$ satisfy the conditions for \Lemma{reduction} with $r_{0}=k^{-1}$, except that $F_{j,k,\ell}$ is not necessarily compact. However, $\cnj{F}_{j,k,\ell}$ is a closed set still satisfying these conditions, it is totally bounded since $\mu$ is doubling, and since $X$ is complete, the Heine-Borel theorem implies $\cnj{F}_{j,k,\ell}$ is compact. Thus, we can apply \Lemma{reduction} to get that $\cnj{F}_{j,k,\ell}$ is rectifiable. Since $F=\bigcup_{j,k,\ell} F_{j,k,\ell}$ we now have that $F$ is also rectifiable.
\end{proof}

The rest of the paper is devoted to proving \Lemma{reduction}, so fix $\mu$, $E$, $\xi_{0}$, and $r_{0}$ as in the lemma.\\

\begin{proof}[Proof of \Lemma{reduction}]

We will require the notion of dyadic cubes on a metric space. This theorem was originally developed by David and Christ (\cite{David88}, \cite{Christ-T(b)}), but the current formulation we take from Hyt\"onen and Martikainen \cite{HM11}.

 \begin{theorem}
Let $X$ be a metric space equipped with a doubling measure $\mu$. Let $X_{n}$ be a nested sequence of maximal $\rho^{n}$-nets for $X$ where $\rho<1/1000$ and let $c_{0}=1/500$. For each $n\in\bZ$ there is a collection $\cD_{n}$ of ``cubes,'' which are Borel subsets of $X$ such that
\begin{enumerate}
\item for every $n$, $X=\bigcup_{\Delta\in \cD_{n}}\Delta$,
\item if $\Delta,\Delta'\in \cD=\bigcup \cD_{n}$ and $\Delta\cap\Delta'\neq\emptyset$, then $\Delta\subseteq \Delta'$ or $\Delta'\subseteq \Delta$,
\item for $\Delta\in \cD$, let $n(\Delta)$ be the unique integer so that $\Delta\in \cD_{n}$ and set $\ell(\Delta)=5\rho^{n(\Delta)}$. Then there is $\zeta_{\Delta}\in X_{n}$ so that
\[B_{X}(\zeta_{\Delta},c_{0}\ell(\Delta) )\subseteq \Delta\subseteq B_{X}(\zeta_{\Delta},\ell(\Delta))\]
and
\[ X_{n}=\{\zeta_{\Delta}: \Delta\in \cD_{n}\}.\]
\end{enumerate}
\label{t:Christ}
\end{theorem}

It is not necessary for there to exist a doubling measure but just that the metric space is geometrically doubling. Moreover, Hyt\"onen and Martikainen use sequences of sets $X_{n}$ slightly more general than maximal  nets, see \cite{HM11} for details. 

Let $X_{n}$ be a nested sequence of maximal $\rho^{n}$-nets for $X$ where $\rho<1/1000$ and $\cD$ the resulting cubes from \Theorem{Christ}. By picking our net points $X_{n}$ appropriately, we may assume that $E\subseteq \Delta_{0}\in \cD$. We recall a lemma \cite{AzzHarm}.

\begin{lemma} \cite[Section 3]{AzzHarm} 
Let $\mu$ be a $C_{\mu}$-doubling measure and let $\cD$ the cubes from \Theorem{Christ} for $X=\supp \mu$ with admissible constants $c_{0}$ and $\rho$. Let $E\subseteq \Delta_{0}\in \cD$ be a Borel set, $M>1$, $\delta>0$, and set
\begin{multline*}
\cP=\{\Delta\subseteq \Delta_{0}: 
\Delta\cap E\neq\emptyset, \exists \; \xi\in B_{X}(\zeta_{\Delta},M\ell(\Delta)) \\
\mbox{ such that } \dist(\xi,E)\geq \delta\ell(\Delta)\}.
\end{multline*}
Then there is $C_{1}=C_{1}(M,\delta,C_{\mu})>0$ so that, for all $\Delta'\subseteq \Delta_{0}$,
\begin{equation}
\sum_{\Delta\subseteq \Delta' \atop \Delta\in \cP} \mu(\Delta)\leq C_{1} \mu(\Delta').
\label{e:sumP}
\end{equation}
\label{l:porous}
\end{lemma}

The theorem is stated in \cite{AzzHarm} in slightly more generality. For the reader's convenience, we provide a shorter proof in the appendix.\\

Let $M,\delta>0$, to be decided later and let $\cP$ be the set from \Lemma{porous} applied to our set $E$. Our goal now is to construct a metric space $Y$ containing $X$, then a curve $\Gamma\subseteq Y$ that contains $E$ as a subset, and then show it has finite length. We will do this by adding bridges through $Y$ between net points around cubes in $\cP$, since these are the cubes where $E$ has large holes and thus potentially has big gaps or disconnections. We don't need the endpoints of these bridges to be in $E$, but their union plus the set $E$ will be connected. We now proceed with the details.\\

Let $\tilde{X}=\bigcup X_{n}$ and equip $C(X) \oplus  \bR^{\tilde{X}\times\tilde{X}}$ (where $\bR^{\tilde{X}\times\tilde{X}}=\prod_{\alpha\in \tilde{X}\times\tilde{X}}\bR$, see \cite[p. 112-117]{M} for the notation) with norm $ |a\oplus b|=\max\{|a|,|b|\}$, where the norm on $\bR^{\tilde{X}\times\tilde{X}}$ is the $\ell^{2}$-norm. 

For  $x,y\in \tilde{X}$ let $[x,y]$ denote the straight line segment between them in $C(X) \oplus  \bR^{\tilde{X}\times\tilde{X}}$, $e_{(x,y)}$ is the unit vector corresponding to the $(x,y)$-coordinate in $\bR^{\tilde{X}\times\tilde{X}}$, and define
\begin{multline*}
[x,y]^{*}  := [x,(x, |x-y|e_{(x,y)})]\cup [y,(y, |x-y|e_{(x,y)})] \\
 \cup [(x, |x-y|e_{(x,y)}),(y, |x-y|e_{(x,y)})] \subseteq  C(X) \oplus  \bR^{\tilde{X}\times\tilde{X}}.\end{multline*}

The set $[x,y]^{*}$ is two segments going straight up from $x$ and $y$ respectively in the $e_{(x,y)}$ direction and a segment connecting the endpoints, thus giving a polygonal curve connecting $x$ to $y$ that hops out of $C(X)$.  Let

\[Y=X\cup \bigcup_{x,y\in \tilde{X}}[x,y]^{*}\]
and define a metric on $Y$ (also denoted by $|\cdot |$) by setting 
\[|x-y|=\inf \sum_{i=1}^{N} |x_{i}-x_{i+1}|\]
where $x_{1}=x$, $x_{N+1}=y$, and for each $i$, $\{x_{i},x_{i+1}\}\subseteq X$ or $\{x_{i},x_{i+1}\}\subseteq [x',y']^{*}$ for some $x',y'\in \tilde{X}$. It is easy to check that the resulting metric space $Y$ is separable and $X$ is a sub metric space in $Y$. Moreover, the following lemma is immediate from the definition of $Y$.

\begin{lemma}\label{l:distE}
Let $F\subseteq X$ be compact and $x,y\in \tilde{X}$. Then
\[\dist([x,y]^{*},F)=\dist(\{x,y\},F).\]
\end{lemma}

We will let 
\[B_{\Delta}:=B_{Y}(\zeta_{\Delta},\ell(\Delta))\supseteq B_{X}(\zeta_{\Delta},\ell(\Delta)).\]

For $\Delta\in \cD_{n}$, let 
\[ \Gamma_{\Delta}=\bigcup \{[x,y]^{*}\subseteq C(X)\oplus \bR^{\tilde{X}\times\tilde{X}}: x,y\in X_{n+n_{0}}\cap MB_{\Delta}\}\]
where $n_{0}$ is an integer we will pick later. Note that $\Gamma_{\Delta}$ is connected and contains $\zeta_{\Delta}$.

Now define
\[\Gamma=E\cup \bigcup_{\Delta\in \cP}\Gamma_{\Delta}.\]

\begin{lemma}\label{l:finite}
$\cH^{1}(\Gamma)<\infty$.
\end{lemma}

\begin{proof}
We first claim that
\begin{equation}\label{e:H(E)<mu(E)}
\cH^{1}(E)\leq 10\mu(E).
\end{equation}
Indeed, let $0<\delta<r_{0}$. Take any countable collection of balls centered on $E$ of radii less than $\delta$ that cover $E$. Since $\mu$ is doubling, we can use the Vitali covering theorem \cite[Theorem 1.2]{H}, to find a countable subcollection of disjoint balls $B_{i}$ with radii $r_{i}<\delta$ centred on $E$ so that $E\subseteq \bigcup 5 B_{i}$. Then
\begin{multline*}
 \cH^{1}_{\delta}(E)
\leq \sum 10r_{i}
\leq 10\sum \mu (B_{i})
\leq 10\mu(\{x\in X:\dist(x,E)<\delta\}).
\end{multline*}
Since $\bigcap_{\delta>0}\{x\in X:\dist(x,E)<\delta\}=E$, sending $\delta\rightarrow 0$ we obtain $\cH^{1}(E)\leq 10\mu(E)$, which proves the claim.

With this estimate in hand, we have
\begin{align*}
\cH^{1}(\Gamma)
& \leq \cH^{1}(E)+\sum_{\Delta\in \cP} \cH^{1}(\Gamma_{\Delta})
\stackrel{\eqn{H(E)<mu(E)}}{\leq } 
10\mu(E)+C\sum_{\Delta\in \cP} \ell(\Delta)\\
& \stackrel{\eqn{mu>r}}{\leq }  10\mu(E)+C\sum_{\Delta\in \cP} \mu(\Delta)
\stackrel{\eqn{sumP}}{\leq} 10\mu(E)+C\mu(\Delta_{0})<\infty
\end{align*}
where $C$ here stands for various constants that depend only on $\delta,M,n_{0},\rho$, and the doubling constant $C_{\mu}$.
\end{proof}

\begin{lemma} \label{l:compact}
$\Gamma$ is compact. 
\end{lemma}

\begin{proof}
To see this, let $x_{n}\in \Gamma$ be any sequence. If $x_{n}\in \Gamma_{\Delta}$ infinitely many times for some $\Delta\in \cP$ or is in $E$ infinitely many times, then since each of these sets are compact, we can find a convergent subsequence with a limit in $\Gamma$. 
Otherwise, $x_{n}$ visits infinitely many $\Gamma_{\Delta}$. Let $x_{n_{j}}$ be a subsequence so that $x_{n_{j}}\in \Gamma_{\Delta_{j}}$ where each $\Delta_{j}\in \cP$ is distinct. Then $\ell(\Delta_{j})\rightarrow 0$, and since $\Delta\cap E\neq\emptyset$ for all $\Delta\in \cP$, $\dist(x_{n_{j}},E)\rightarrow 0$. Pick $x_{n_{j}}'\in E\cap \Delta_{j}$. Since $E$ is compact, there is a subsequence $x_{n_{j_{k}}}'$ converging to a point in $E$, and $x_{n_{j_{k}}}$ will have the same limit. We have thus shown that any sequence in $\Gamma$ has a convergent subsequence, which implies $\Gamma$ is compact.
\end{proof}

\begin{lemma}\label{l:schul}
A compact connected metric space $X$ of finite length can be parametrised by a Lipschitz image of an interval in $\bR$, that is, $X=f([0,1])$ where $f:[0,1]\rightarrow X$ is Lipschitz.
\end{lemma}

A proof of this fact for Hilbert spaces is given in \cite[Corollary 3.7]{Schul-TSP}, but the same proof works in our setting, so we omit it. Hence, to show that $\Gamma$ (and hence $E$) is rectifiable, all that remains to show is that $\Gamma$ is connected.

\begin{lemma}\label{l:connected}
The set $\Gamma$ is connected.
\end{lemma}

\begin{proof}
Suppose for the sake of a contradiction that there exist two open and disjoint sets $A$ and $B$ that cover $\Gamma$ and set $\Gamma_{A}=\Gamma\cap A$ and $\Gamma_{B}=\Gamma\cap B$. Suppose without loss of generality that $\Gamma_{\Delta_{0}}\subseteq \Gamma_{A}$, which we may do since $\Gamma_{\Delta_{0}}$ is connected. We sort the proof into a series of steps.

\begin{enumerate}[(a)]

\item $\Gamma_{B}\subseteq 2B_{\Delta_{0}}$. To see this, suppose instead that there is $z\in \Gamma_{B}\backslash 2B_{\Delta_{0}}$. Then $z\in [x,y]^{*}\subseteq \Gamma_{\Delta}$ for some $\Delta\in \cP$. Moreover, $\dist(z,\{x,y\})\leq 2|x-y|\leq 4M\ell(\Delta)$ since $x,y\in MB_{\Delta}$. Since $\zeta_{\Delta}\in \Delta\subseteq \Delta_{0}$ and $x\in MB_{\Delta}$, we get 
\begin{align*}
 \ell(\Delta_{0}) 
 & \leq \dist(z,B_{\Delta_{0}})\leq |z-x|+\dist(x,B_{\Delta_{0}})\leq 4M\ell(\Delta)+M\ell(\Delta)\\
 & =5M\ell(\Delta).\end{align*}

For $n_{0}$ large enough so that $5M\rho^{n_{0}}<1$, this implies $\zeta_{\Delta}\in X_{n+n_{0}}\cap MB_{\Delta_{0}}$ and so $\Gamma_{\Delta}\cap \Gamma_{\Delta_{0}}\neq\emptyset$. Hence, $\Gamma_{\Delta}\subseteq \Gamma_{A}$ since $\Gamma_{\Delta}$ is connected, contradicting that $z\in \Gamma_{B}$. This proves the claim.

\item The open sets $A'=A\cup (\cnj{4B_{\Delta_{0}}})^{c}$ and $B'=B\cap 2B_{\Delta_{0}}$ are disjoint and cover $\Gamma$. First, observe that
\begin{align*}
A'\cap B' & =(A\cap B\cap 2B_{\Delta_{0}})\cup ((\cnj{4B_{\Delta_{0}}})^{c}\cap B\cap 2B_{\Delta_{0}})
\\
& \subseteq (A\cap B)\cup ( (\cnj{4B_{\Delta_{0}}})^{c}\cap 2B_{\Delta_{0}})=\emptyset.
\end{align*}

Moreover, by part (a),
\[
\Gamma\cap (A'\cup B')
\supseteq \Gamma_{A}\cup (\Gamma_{B}\cap 2B_{\Delta_{0}})
=\Gamma_{A}\cup \Gamma_{B}=\Gamma\]
which completes the proof of this step. 

\item Set $\Gamma_{A'}=\Gamma\cap A'$ and $\Gamma_{B'}=\Gamma\cap B'$. These sets are disjoint by part (b) and hence they are compact since $\Gamma$ was compact. We define new open sets
\[A''=(\cnj{4B_{\Delta_{0}}})^{c} \cup \bigcup_{\xi\in \Gamma_{A'}} B_{Y}(\xi,\dist(\xi,\Gamma_{B'})/2)\]
and 
\[ B''=\bigcup_{\xi\in \Gamma_{B'}} B_{Y}(\xi,\dist(\xi,\Gamma_{A'})/2).\]

We claim these sets are disjoint. Suppose there is $z\in A''\cap B''$. Then $z\in B_{Y}(\xi,\dist(\xi,\Gamma_{A'})/2)$ for some $\xi\in \Gamma_{B'}$. If we also have $z\in B_{Y}(\xi',\dist(\xi',\Gamma_{B'})/2)$ for some $\xi'\in \Gamma_{A'}$, then
\begin{align*}
 \max\{\dist(\xi,\Gamma_{B'}),\dist(\xi',\Gamma_{A'})\}
 & \leq  |\xi-\xi'|
 \leq |\xi-z|+|z-\xi| \\
  & <\frac{\dist(\xi,\Gamma_{B'})}{2}+\frac{\dist(\xi',\Gamma_{A'})}{2},
  \end{align*}
which is a contradiction, so we must have $z\in (\cnj{4B_{\Delta_{0}}})^{c}$. Since $\xi\in \Gamma_{B'}$, we know $\xi\in 2B_{\Delta_{0}}$ by part (a), and $\zeta_{\Delta_{0}}\in \Gamma_{\Delta_{0}}\subseteq \Gamma_{A'}$ implies $\dist(\xi,\Gamma_{A'})\leq 2\ell(\Delta_{0})$. Hence, 
\[ B_{Y}(\xi,\dist(\xi,\Gamma_{A'})/2)\subseteq B_{Y}(\xi,\ell(\Delta_{0}))\subseteq B_{Y}(\zeta_{\Delta_{0}},3\ell(\Delta_{0}))=3B_{\Delta_{0}},\]
which proves the claim. 

\item Note that $X\backslash (A''\cup B'')$ is nonempty since $X$ is connected and $A''$ and $B''$ are disjoint open sets. Moreover, $X\backslash (A''\cup B'')\subseteq \cnj{4B_{\Delta_{0}}}$ and hence a bounded set; since $X$ is a doubling metric space, $X\backslash (A''\cup B'')$  is in fact totally bounded and thus compact by the Heine-Borel theorem. This implies we can find a point 
\[z\in X\backslash (A''\cup B'') \subseteq \cnj{4B_{\Delta_{0}}}\] 
of maximal distance from the compact set $\Gamma$.

\item Let $\xi\in E$ be the closest point to $z$ and $\Delta$ the smallest cube containing $\xi$ so that $z\in 5B_{\Delta}$; since $z\in \cnj{4B_{\Delta_{0}}}\subseteq 5B_{\Delta_{0}}$, this is well defined. We claim $\Delta\in \cP$. If $\Delta_{1}$ denotes the child of $\Delta$ that contains $\xi$, then $z\not\in 5B_{\Delta_{1}}$, and so
\begin{align}\label{e:distzE}
\dist(z,E) & =|\xi-z|\geq |z-\zeta_{\Delta_{1}}|-|\zeta_{\Delta_{1}}-\xi|
\geq 5\ell(\Delta_{1})-\ell(\Delta_{1})  \notag \\
& =4\rho\ell(\Delta).\end{align}
Thus, for $M>10$, $B_{X}(z,4\rho\ell(\Delta))\subseteq MB_{\Delta}\backslash E$, so if $\delta<4\rho$, then $\Delta\in \cP$, which proves the claim. 

\item Since $\Delta\in \cP$, $X_{n(\Delta)+n_{0}}$ is a maximal $\rho^{n(\Delta)+n_{0}}$-net, \[\rho^{n(\Delta)+n_{0}}<\rho^{n_{0}}\ell(\Delta)<\ell(\Delta),\] 
and $z\in 5B_{\Delta}$, we can find
\begin{align}
\zeta & \in X_{n(\Delta)+n_{0}}\cap B_{X}(z,\rho^{n(\Delta)+n_{0}}) \label{e:zeta1} \\
& \subseteq X_{n(\Delta)+n_{0}}\cap B_{X}(\zeta_{\Delta},5\ell(\Delta)+\rho^{n(\Delta)+n_{0}}) \notag \\
& \subseteq X_{n(\Delta)+n_{0}} \cap B_{X}(\zeta_{\Delta},6\ell(\Delta))
 \subseteq \Gamma_{\Delta}
 \label{e:zeta}
\end{align}
where the last containment follows if we assume $M>6$.

Since $\Gamma_{\Delta}$ is connected and $A'$ and $B'$ are disjoint open sets, we may without loss of generality suppose $\Gamma_{A'}\supseteq \Gamma_{\Delta}$ and let $\zeta'\in \Gamma_{B'}$ be the closest point to $\zeta$. Then 
\begin{equation}\label{e:z-zeta>zeta-zeta}
|z-\zeta|\geq |\zeta-\zeta'|/2=\dist(\zeta,\Gamma_{B'})/2
\end{equation} 
since otherwise would imply $z\in B_{Y}(\zeta,\dist(\zeta,\Gamma_{B'})/2)\subseteq A''$, contradicting that $z\in X\backslash (A''\cup B'')$. 

We may assume $\zeta'\in \Gamma_{\Delta'}$ for some $\Delta'\in \cP$ and we assume $\Delta'$ is the largest such cube for which this happens. Note that this implies $\Gamma_{\Delta'}\subseteq \Gamma_{B'}$ since $\zeta'\in \Gamma_{B'}\cap \Gamma_{\Delta'}$ and $\Gamma_{\Delta'}$ is connected. By \Lemma{distE} with $F=\{\zeta\}$, we can assume $\zeta'\in X$, and so $\zeta'\in X_{n(\Delta')+n_{0}}\cap MB_{\Delta'}$.

\item We claim that $n(\Delta)+1\leq n(\Delta')\leq n(\Delta)+2$. Note that since 
\begin{equation}\label{e:fml}
5\rho^{n(\Delta)+n_{0}}\leq \ell(\Delta)\rho^{n_{0}}\leq \rho\ell(\Delta)<\ell(\Delta), 
\end{equation}
we have
\begin{align}\label{e:zeta'-xi}
|\zeta'-\zeta_{\Delta}|
& \leq |\zeta'-\zeta|+|\zeta-\zeta_{\Delta} | 
 \stackrel{\eqn{zeta}\atop \eqn{z-zeta>zeta-zeta} }{<} 2|\zeta-z|+6\ell(\Delta)\notag \\
& \stackrel{\eqn{zeta1}}{<} 2\rho^{n(\Delta)+n_{0}}+6\ell(\Delta)
\stackrel{\eqn{fml}}{\leq} 8\ell(\Delta).
\end{align}

Thus, for $M>8$, we must have $n(\Delta')> n(\Delta)$; otherwise, since $\xi\in \Delta\subseteq B_{\Delta}$, we would have
\[
\zeta'\in X_{n(\Delta')+n_{0}}\cap 8B_{\Delta}
\subseteq X_{n(\Delta)+n_{0}}\cap MB_{\Delta}\subseteq \Gamma_{\Delta}\]
so that $\Gamma_{\Delta}\cap \Gamma_{\Delta'}\neq\emptyset$, which implies $\Gamma_{A'}\cap \Gamma_{B'}\neq\emptyset$, a contradiction. Thus, $\ell(\Delta')< \ell(\Delta)$, which proves the first inequality in the claim.  

Note this implies $\ell(\Delta')\leq \rho \ell(\Delta)$. Let $\xi'\in \Delta'\cap E$ (which exists since $\Delta'\in \cP$). Since $\zeta'\in MB_{\Delta'}$ we have
\begin{align*}
4\rho\ell(\Delta) 
& \stackrel{\eqn{distzE}}{\leq} \dist(z,E) 
 \leq |\xi'-z|\\
& \leq |\xi'-\zeta_{\Delta'}|+|\zeta_{\Delta'}-\zeta'|+|\zeta'-\zeta|+|\zeta-z|\\
& \stackrel{\eqn{z-zeta>zeta-zeta} }{\leq} \ell(\Delta')+M\ell(\Delta')+2|\zeta-z|+|\zeta-z|\\
& \stackrel{\eqn{zeta}}{\leq} (M+1)\ell(\Delta')+3\rho^{n(\Delta)+n_{0}}\\
& \stackrel{\eqn{fml}}{\leq} (M+1)\ell(\Delta')+\rho\ell(\Delta)
\end{align*}
and so
\[
\frac{3\rho}{M+1} \ell(\Delta)\leq \ell(\Delta').\]
Thus, $\rho<\frac{3}{M+1}$ implies $\rho^{2}\ell(\Delta)\leq \ell(\Delta')$, and so $n(\Delta')\leq n(\Delta)+2$, which finishes the claim.

\item Now we'll show that $\Gamma_{\Delta}\cap \Gamma_{\Delta'}\neq\emptyset$. Observe that

\begin{align}
|\zeta_{\Delta}-\zeta_{\Delta'}|
& \leq |\zeta_{\Delta}-\zeta'|+|\zeta'-\zeta_{\Delta'}|
 \stackrel{\eqn{zeta'-xi}}{\leq} 8\ell(\Delta)+M\ell(\Delta') \notag \\
& \leq (8+M\rho)\ell(\Delta)< M\ell(\Delta)
\label{e:M+9}
\end{align}
if $\rho^{-1}>M>9$. Since $n(\Delta')\leq n(\Delta)+2$, we have that $\zeta_{\Delta'}\in X_{n(\Delta)+n_{0}}\cap MB_{\Delta}$ for $n_{0}\geq 2$ and so $\zeta_{\Delta'}\in \Gamma_{\Delta}$. But $\zeta_{\Delta'}\in X_{n(\Delta')+n_{0}}\cap MB_{\Delta'}\subseteq \Gamma_{\Delta'}$, thus  $\Gamma_{\Delta}\cap \Gamma_{\Delta'}\neq\emptyset$, which proves the claim. 

This gives us a grand contradiction since $\Gamma_{\Delta}\subseteq \Gamma_{A'}$ and $\Gamma_{\Delta'}\subseteq \Gamma_{B'}$, and we assumed these sets to be disjoint.

\end{enumerate}
\end{proof}

Combining Lemmas \ref{l:finite}, \ref{l:compact}, \ref{l:schul}, and \ref{l:connected}, we have now shown that $E$ is contained in the Lipschitz image of an interval in $\bR$. This completes the proof of \Lemma{reduction}.
\end{proof}

\section{Appendix: Proof of \Lemma{porous}}

For $\Delta\in \cD$,  define $B_{\Delta}=B_{X}(\zeta_{\Delta},\ell(\Delta))$. For $\Delta\in \cP$, let $\xi_{\Delta}\in MB_{\Delta}$ be such that $\dist(\xi,E)\geq \delta\ell(\Delta)$. Let $\cM$ be the collection of maximal cubes for which $2B_{\Delta}\subseteq E^{c}$ and $\tilde{\Delta}\in \cM$ be the largest cube containing $\xi_{\Delta}$. Then if $\tilde{\Delta}^{1}$ denotes the parent cube of $\tilde{\Delta}$, $2B_{\tilde{\Delta}^{1}}\cap E\neq\emptyset$, and so 
\begin{equation}
\delta \ell(\Delta)\leq \dist(\xi_{\Delta},E)\leq \diam 2B_{\tilde{\Delta}^{1}} \leq 4\ell(\tilde{\Delta}^{1})=\frac{4}{\rho}\ell(\tilde{\Delta})
\label{e:tildedelta}
\end{equation}
Moreover, 
\begin{equation}
\ell(\tilde{\Delta})\leq \frac{2M}{c_{0}}\ell(\Delta)
\label{e:tildedelta2}
\end{equation}
for otherwise $\tilde{\Delta}\supseteq c_{0} B_{\tilde{\Delta}}\supseteq MB_{\Delta}\supseteq \Delta$ and since $\Delta\cap E\neq\emptyset$,  this means $2B_{\tilde{\Delta}}\cap E\neq\emptyset$, contradicting our definition of $\tilde{\Delta}$.

%    \Claim When $\Delta'\subseteq MB_{\Delta}$ and $\Delta\subseteq \Delta_{0}$,
%    \begin{equation}
%    (\ell(\Delta')/\ell(\Delta))^{\beta_{0}}\lec_{M,C_{\mu}} \mu(\Delta' )/\mu(\Delta).
%    \label{e:b0}
%    \end{equation}

Let $N_{\Delta}$ be such that 
\begin{equation}
\label{e:2^N}
2^{N_{\Delta}}c_{0}\ell(\tilde{\Delta})>2M\ell(\Delta)>2^{N_{\Delta}-1}c_{0}\ell(\tilde{\Delta}).
\end{equation}
Then $2^{N_{\Delta}}c_{0}B_{\tilde{\Delta}}\supseteq MB_{\Delta}$, and $2^{N_{\Delta}}<\frac{4M\ell(\Delta)}{c_{0}\ell(\tilde{\Delta})}$, so that 
\begin{equation}
N_{\Delta}<\log_{2}\ps{\frac{4M\ell(\Delta)}{c_{0}\ell(\tilde{\Delta})}}.
\label{e:ndelta}
\end{equation}
Thus
\begin{align}
\frac{\mu(\tilde{\Delta})}{\mu(\Delta)}
& \geq \frac{\mu(c_{0}B_{\tilde{\Delta}})}{\mu(\Delta)}
\stackrel{\eqn{doubling}}{\geq} \frac{\mu(2^{N_{\Delta}}c_{0}B_{\tilde{\Delta}}) }{C_{\mu}^{N_{\Delta}} \mu(\Delta)}
\stackrel{\eqn{2^N}}{ \geq} \frac{\mu(MB_{\Delta})}{C_{\mu}^{N_{\Delta}} \mu(\Delta)}  \notag \\
&   \stackrel{\eqn{ndelta}}{\geq} C_{\mu}^{\log_{2}\frac{c_{0}}{4M}}\ps{\frac{\ell(\tilde{\Delta})}{\ell(\Delta)}}^{\log_{2}C_{\mu}}
 \stackrel{\eqn{tildedelta}}{\geq } 
C_{\mu}^{\log_{2}\frac{c_{0}}{4M}}\ps{\frac{4}{\rho}}^{\log_{2} C_{\mu}}=: a
\label{e:a}
\end{align}

Since $\mu$ is doubling and $\Delta$ and $\Delta'$ are always of comparable sizes by \eqn{tildedelta} and \eqn{tildedelta2}, there is $b$ depending on $M,\delta,\rho,c_{0}$ and $C_{\mu}$ such that at most $b$ many cubes $\Delta\in \cM$ with $\tilde{\Delta}=\Delta'$ for some fixed $\Delta'$. Hence, for $\Delta'\subseteq \Delta_{0}$ with $\Delta\cap E\neq\emptyset$, 
\begin{align*}
\sum_{\Delta\subseteq \Delta' \atop \Delta\in \cP}\mu(\Delta)
& \stackrel{\eqn{a}}{\leq} \sum_{\Delta\subseteq \Delta' \atop \Delta\in \cP} a\mu(\tilde{\Delta})
=\sum_{\Delta'\in \cM\atop \Delta\subseteq MB_{\Delta_{0}}}\sum_{{\Delta\subseteq \Delta' \atop \Delta\in \cP}\atop \tilde{\Delta}=\Delta'} a\mu(\tilde{\Delta})
\leq \sum_{\Delta'\in \cM \atop \Delta\subseteq MB_{\Delta_{0}}} ab\mu(\Delta')\\
& \leq ab\mu(MB_{\Delta_{0}}\backslash E)
\leq ab\mu(MB_{\Delta_{0}})
\stackrel{\eqn{doubling}}{\leq} abC_{\mu}^{\log_{2} \frac{M}{c_{0}}+1}\mu(c_{0}B_{\Delta_{0}})\\
& \leq abC_{\mu}^{\log_{2} \frac{M}{c_{0}}+1}\mu(\Delta_{0})
\end{align*}

This finishes the proof of \Lemma{porous}.

\def\cprime{$'$}
\providecommand{\bysame}{\leavevmode\hbox to3em{\hrulefill}\thinspace}
\providecommand{\MR}{\relax\ifhmode\unskip\space\fi MR }
\providecommand{\MRhref}[2]{%
  \href{http://www.ams.org/mathscinet-getitem?mr=#1}{#2}
}
\providecommand{\href}[2]{#2}

\end{document}